\documentclass[a4paper,10pt,twoside]{article}

\usepackage{tensor}
\usepackage[utf8]{inputenc}
\usepackage[UKenglish]{babel}
\usepackage{cite}
\usepackage{amssymb}
\usepackage{amsmath}
\usepackage{mathrsfs}
\usepackage{amsthm}
\usepackage[T1]{fontenc}
\usepackage{accents}
\usepackage[cm]{fullpage}
\usepackage{graphicx}
\usepackage{tensor}
\usepackage{relsize}
\usepackage{appendix}
\usepackage[affil-it]{authblk}
\usepackage{tensor}
\usepackage{verbatim}
\usepackage{comment}

\allowdisplaybreaks

\usepackage{mathtools}
\mathtoolsset{showonlyrefs,showmanualtags}

\RequirePackage{lastpage}
\RequirePackage{latexsym}
\makeatletter
\ifx\@NODS\undefined\RequirePackage{dsfont}\fi
\ifx\@NOAMS\undefined\RequirePackage{amsmath,amsfonts,amssymb,amsthm}\fi
\makeatother

\RequirePackage{geometry}
\RequirePackage{bera}
\RequirePackage[pdftex,pagebackref=false]{hyperref}
\hypersetup{pdfborder=0 0 0}
\hypersetup{pdfstartview={FitH}}

\title{{\Large \bfseries{Functional inequalities for Feynman-Kac semigroups}}}

\author{ }
\author{James Thompson}

\date{\today}

\makeatletter
\def\@MRExtract#1 #2!{#1}
\newcommand{\MR}[1]{
  \xdef\@MRSTRIP{\@MRExtract#1 !}%
  \href{http://www.ams.org/mathscinet-getitem?mr=\@MRSTRIP}{MR-\@MRSTRIP}}
\makeatother

\makeatletter
\renewenvironment{thebibliography}[1]{%
  \section*{\refname
    \@mkboth{\MakeUppercase\refname}{\MakeUppercase\refname}}%
  \phantomsection%
  \addcontentsline{toc}{section}{\refname}%
  \list{\@biblabel{\@arabic\c@enumiv}}%
  {\settowidth\labelwidth{\@biblabel{#1}}%
    \small%
    \setlength{\labelsep}{0.4em}%
    \setlength{\leftmargin}{\labelwidth}%
    \addtolength{\leftmargin}{\labelsep}%
    \setlength{\itemsep}{-.25em}%
    \@openbib@code
    \usecounter{enumiv}%
    \let\p@enumiv\@empty
    \renewcommand\theenumiv{\@arabic\c@enumiv}}%
  \sloppy\clubpenalty4000\@clubpenalty\clubpenalty\widowpenalty4000%
  \sfcode`\.\@m}{%
  \def\@noitemerr{%
    \@latex@warning{Empty `thebibliography' environment}}%
  \endlist}
\makeatother

\makeatletter
\ifx\@NODS\undefined%

\let\mathbb=\mathds
\else%
\fi
\makeatother

\DeclareMathOperator{\divv}{div}

\DeclareMathOperator{\End}{End}

\DeclareMathOperator{\tr}{tr}

\DeclareMathOperator{\Aut}{Aut}
\DeclareMathOperator{\loc}{loc}

\def\<{\langle}
\def\>{\rangle}

\def\Hess{\mathop{\rm Hess}}
\def\Ric{\mathop{\rm Ric}}
\def\id{\mathop{\rm id}}
\def\V{\mathbb V}

\newcommand{\E}{\mathbb{E}}

\newtheorem{theorem}{Theorem}[section]

\newtheorem{proposition}[theorem]{Proposition}

\begin{document}

\maketitle

\begin{abstract}%
   \noindent%
   {Using the tools of stochastic analysis, we prove various gradient estimates and Harnack inequalities for Feynman-Kac semigroups with possibly unbounded potentials. One of the main results is a derivative formula which can be used to characterize a lower bound on Ricci curvature using a potential.}\\[1em]%
   {\footnotesize%
     \par%
     \par
   }
 \end{abstract}

\section{Introduction}

Suppose $M$ is a geodesically and stochastically complete and connected Riemannian manifold of dimension $n$, with Levi-Civita connection $\nabla$ and Laplace-Beltrami operator $\Delta$. Suppose $V \in L^2_{\loc}(M)$ is bounded below. Then the operator $H:=-\frac{1}{2}\Delta +V$ is essentially self-adjoint \cite{BMS} and bounded below. Essential self-adjointness means that $H$ can be uniquely extended from the domain $C_0^\infty(M)$ of smooth functions with compact support to a self-adjoint operator $H$ with maximal domain $D(H)=\lbrace f : f, Hf \in L^2(M)\rbrace$. The corresponding semigroup generated by $-H$ consists of bounded self-adjoint operators on $L^2(M)$. Under certain conditions on $V$ the semigroup is continuous and given by the Feynman-Kac formula
\begin{equation}
P^V_tf(x) = \E\left[ \V^x_t f(X_t(x))\right].
\end{equation}
Here $X(x)$ is a Brownian motion on $M$ starting at $x$ and
\begin{equation}
\V_t^x := e^{-\int_0^t V(X_s(x)) ds}
\end{equation}
for $x \in M$ and $t \geq 0$. A class of potentials that is of particular interest is those $V$ for which
\begin{equation}
\lim_{t\downarrow 0} \sup_{x\in M} \E\left[\int_0^t |V(X_s(x))| d s \right] = 0.
\end{equation}
Such $V$ is said to belong to the Kato class, first introduced by Kato in \cite{Kato1972} using an equivalent definition. The local Kato class consists of those potentials $V$ for which $\mathbf{1}_K V$ belongs to the Kato class for any compact set $K \subset M$. This is a very large class of potentials, encompassing nearly all physically relevant phenomena, for which one can still expect the Feynman-Kac formula to make sense pointwise for all $x \in M$. We refer to Aizenman and Simon \cite{AizenmanSimon1982} and Simon \cite{Simon1982} for a comprehensive account of all this, and to the more recent work of G\"{u}neysu \cite{Guneysu2012} for generalizations to Schr\"{o}dinger type operators on vector bundles.

Our focus will be on the derivative, or gradient, of the Feynman-Kac semigroup $P^V_tf$. For the heat semigroup $P_tf$, a probabilistic formula for the derivative $dP_tf$, not involving the derivative of $f$, was proved by Bismut in \cite{Bismut1984} using techniques from Malliavin calculus. A transparent generalization of this, using an argument based on martingales, was later given by Elworthy and Li in \cite{ElworthyLi1994} and developed further by Thalmaier in \cite{Thalmaier1997}. Thalmaier's formulation is based on local martingales and allows to localize the contribution of Ricci curvature. A formula for $dP^V_tf$ was proved by Elworthy and Li in \cite{ElworthyLi1994}, some applications of which were recently explored in \cite{LiThompson2018}.

The present work is a continuation of the author's recent article \cite{hessianpaperone}, in which localized versions of the derivative formula were proved, together with an extension to the Hessian. In all cases, these derivative formulae involve the derivative of $V$. Therefore, as in \cite{ElworthyLi1994} and \cite{hessianpaperone}, we will assume throughout that $V$ is continuously differentiable.

In Section \ref{sec:gradests}, we start by using the derivative formula from \cite{hessianpaperone} to obtain explicit gradient estimates for $P^V_tf$. These estimates are then used to derive Theorem \ref{thm:derivform}, which states that if
\begin{equation}\label{eq:conditionvintro}
\sup_{x \in M} \E \left[ \int_0^t |dV|(X_s(x))ds\right] < \infty
\end{equation}
with the Ricci curvature of $M$ bounded below, then
\begin{equation}
(dP^V_tf)(v) = \E\left[ \V_t^x ((df)_{X_t(x)}(//_tQ_t v) -f(X_t(x))\int_0^t (dV)_{X_s(x)}(//_s Q_s v)ds)\right]
\end{equation}
for all $f\in C_b^1(M)$, $x \in M$, $v \in T_xM$ and $t \geq 0$. The composition $Q//^{-1}$ is the damped parallel transport along the paths of $X(x)$, determined by the solution to equation \eqref{eq:eqforQ} below, and $C_b^1(M)$ denotes the space of all bounded continuously differentiable functions on $M$ with bounded derivative. Note that condition \eqref{eq:conditionvintro} is satisfied if $|dV|$ belongs to the Kato class.

One application of this derivative formula is that it allows to characterize a lower bound on Ricci curvature using the potential $V$, as explained in Section \ref{sec:characters}. For example, suppose $V\in C^1(M)$ is bounded below with $\nabla V$ bounded and $K \in \mathbb{R}$. Then, according to Theorem \ref{thm:gradestbas}, the following are equivalent:
\begin{itemize}
\item[\normalfont{(1)}] $\Ric \geq 2K$;
\item[\normalfont{(2)}] if $f\in C_b^1(M)$ then
\begin{equation}
|\nabla P^V_tf| \leq e^{-Kt}P^V_t|\nabla f| + \|\nabla V\|_\infty\left(\frac{1-e^{-Kt}}{K}\right)P_t^V |f|
\end{equation}
for all $t \geq 0$.
\end{itemize}
In particular, if the above gradient estimate is verified for some suitable $V$ then it must hold for all such $V$.

The derivative formula stated above is also applied in Section \ref{sec:harn} to prove a Harnack inequality, given by Theorem \ref{thm:harnack}. Afterwards, in Section \ref{sec:shiftharn}, we prove the counterpart to the Bismut formula for differentiation inside the semigroup, namely Theorem \ref{thm:diffformula}, and use it to derive two further derivative estimates, Propositions \ref{prop:firstdivest} and \ref{prop:seconddivest}, with corresponding shift-Harnack inequalities, Theorems \ref{thm:shiftharnone} and \ref{thm:shiftharntwo}. Shift-Harnack inequalities were introduced by Wang in \cite{Wangharn}. Exploring such inequalities in the present setting was motivated by a question posed by Feng-Yu Wang during a workshop at the University of Swansea in 2017.

\section{Uniform gradient estimates}\label{sec:gradests}

In \cite{PW}, Priola and Wang derived uniform gradient estimates for semigroups generated by second order elliptic operators with irregular and unbounded coefficients and bounded, measurable potentials $V$. They did not require local H\"{o}lder continuity of the coefficients, instead replacing the gradient of the semigroup with the Lipschitz constant. Although we will only consider the particular operator $-\frac{1}{2}\Delta +V$, we do allow for the potential to be unbounded.

Let us first briefly consider the case in which the potential $V$ is bounded. In this case, we have the following gradient estimate, which depends only on the supremum norm of $V$:

\begin{proposition}
Suppose $V\in C^1(M)$ is bounded with $\Ric \geq 2K$. Then for all bounded measurable functions $f$ we have
\begin{equation}\label{eq:vocestimate}
\|dP_t^Vf\|_\infty \leq \sqrt{\frac{2}{\pi t}} e^{K^-t} \left(1 + 2t\|V\|_\infty e^{-t \inf V }\right) \|f\|_\infty\\
\end{equation}
for all $t>0$.
\end{proposition}

\begin{proof}
According to the variation of constants formula we have
\begin{equation}
P_t^Vf = P_tf - \int_0^t P_{t-s}(V P^V_sf)ds
\end{equation}
and therefore
\begin{equation}\label{eq:estunibnd}
|dP_t^Vf| \leq |dP_tf| + \int_0^t |dP_{t-s}(V P^V_sf)| ds.
\end{equation}
It is well known that
\begin{equation}
|dP_tf| \leq  \sqrt{\frac{2}{\pi t}} e^{K^-t} \|f\|_\infty
\end{equation}
and similarly for the integrand. The result therefore  follows by substituting these estimates into \eqref{eq:estunibnd}, integrating the latter. 
\end{proof}

We now turn our attention to the case of unbounded $V$. Let us recall the Bismut formula for $P^V_tf$ that was proved in \cite{hessianpaperone}. For this, denote by $//$ the stochastic parallel transport along the paths of $X(x)$ and by $B$ the anti-development of $//$ to $T_xM$. Then $B$ is a Brownian motion on $T_xM$ starting at the origin. Denote by $Q$ the $\End(T_xM)$-valued solution to the ordinary differential equation
\begin{equation}\label{eq:eqforQ}
\frac{d}{dt} Q_t  = -\frac12//_t^{-1} {\Ric}^\sharp //_tQ_t
\end{equation}
with $Q_0 = \id_{T_xM}$. The composition $Q//^{-1}$ is called the damped parallel transport. Now suppose $D$ is a regular domain in $M$ and denote by $\tau$ the first exit time of $X(x)$ from $D$. In \cite{hessianpaperone} it was proved, for a bounded measurable function $f$ and $x \in D$, that there is the following local derivative formula:
\begin{equation}\label{eq:bismutloc}
(dP^V_tf)_x  = -\mathbb{E}\left[ \V^x_t f(X_t(x))\left(\int_0^{t} \langle Q_s\dot{h}_s,dB_s\rangle + dV(//_s Q_sh_s)ds\right)\right]
\end{equation}
for all $t>0$, where $h$ is any bounded adapted process with paths belonging to the Cameron-Martin space $L^{1,2}([0,t];\Aut(T_{x}M))$ such that $h_0=1$, $h_s=0$ for $s \geq \tau \wedge t$ with $\E[\int_0^{\tau \wedge t} |\dot{h}_s|^2ds] < \infty$. Such $h$ can always be constructed, as shown for example in \cite{Ch-Tho-Tha:2017a}. It follows that if the Ricci curvature is bounded below, in the sense that $$\inf\lbrace \Ric(v,v) : v\in TM, |v|=1\rbrace > -\infty,$$ and if
\begin{equation}
\kappa_V(t):=\sup_{x \in M} \E \left[ \int_0^t |dV|(X_s(x))ds\right] < \infty
\end{equation}
then $dP^Vf$ is uniformly bounded on $[\epsilon,t]\times M$ for all $\epsilon>0$ and $t>0$. Arguing as in the proof of \cite[Theorem~7]{hessianpaperone}, it follows that in this case there is the following non-local, but explicit, version of the derivative formula:
\begin{equation}\label{eq:bismut}
(dP^V_tf)_x  = \frac{1}{t}\mathbb{E}\left[ \V^x_t f(X_t(x))\left(\int_0^{t} \langle Q_s,dB_s\rangle - (t-s) dV(//_s Q_s)ds\right)\right]
\end{equation}
for all $t>0$. Using formula \eqref{eq:bismut}, we can quickly derive some simple gradient estimates. For convenience, and without loss of generality, we will assume that $V$ is non-negative. Note also that in the sequel, $K$ will always denote a real-valued constant.

\begin{theorem}\label{thm:linfest}
Suppose $V\in C^1(M)$ is non-negative with $\kappa_V(t)<\infty$ and $\Ric \geq 2K$. Then for all $f \in L^\infty(M)$ we have
\begin{equation}
\|dP^V_t f\|_\infty \leq e^{K^-t} \left( \sqrt{\frac{2}{\pi t}}  + \kappa_V(t) \right)\|f\|_\infty
\end{equation}
for all $t>0$.
\end{theorem}

\begin{proof}
Since for each $t>0$ the stochastic integral $\int_0^t \langle Q_s,dB_s\rangle$ is a centred Gaussian random variable with variance $\int_0^t \| Q_s\|^2 ds \leq t e^{2K^-}$ and since for centred Gaussian random variable with variance $1$ we have $\E[|X|] = \sqrt{\frac{2}{\pi}}$, we have by formula \eqref{eq:bismut}, that
\begin{equation}
\begin{split}
|dP^V_t f|(x) \leq\text{ }& \|f\|_\infty \frac{1}{t}\left( \E \left[ \bigg|\int_0^t \langle Q_s,dB_s\rangle \bigg|\right] + e^{K^-t}\E \left[ \int_0^t (t-s)|dV|(X_s(x))ds\right] \right)\\
\leq \text{ }& e^{K^- t} \|f\|_\infty \left( \sqrt{\frac{2}{\pi t}} +\E \left[ \int_0^t |dV|(X_s(x))ds\right] \right)
\end{split}
\end{equation}
from which the result follows, by the definition of $\kappa_V(t)$. 
\end{proof}

For $p>1$ we can similarly obtain an estimate using the $L^p$ norm (as opposed to the $L^\infty$ norm), for which we should assume that
\begin{equation}
\kappa_{V,q}(t):=\sup_{x \in M} \E\left[\left(\int_0^t |dV|(X_s(x))ds\right)^q\right]^\frac{1}{q} < \infty
\end{equation}
where $q$ is the conjugate of $p$.

\begin{theorem}\label{thm:lpest}
Suppose $V\in C^1(M)$ is non-negative with $\Ric \geq 2K$. Suppose $p>1$, set $q=p/(p-1)$ and assume $\kappa_{V,q}(t)<\infty$. Then for all $f \in L^p(M)$ we have
\begin{equation}
\|dP^V_t f\|_{p} \leq e^{K^-t}  \left(\frac{C_q^{1/q}}{\sqrt{t}}  + \kappa_{V,q}(t) \right)\|f\|_{p}
\end{equation}
for all $t>0$, where $C_q$ is the constant from the Burkholder-Davis-Gundy inequality.
\end{theorem}

\begin{proof}
By the Burkholder-Davis-Gundy inequality, we see that
\begin{align}
\E\left[ \bigg\vert \int_0^{t} \langle Q_s,dB_s\rangle \bigg\vert^q\right] \leq \, & \E\left[ \left( \sup_{0 \leq s \leq t} \bigg\vert \int_0^{s} \langle Q_r,dB_r\rangle \bigg\vert \right)^q\right]\\
\leq \, & C_q \E\left[ \left( \int_0^t \| Q_s\|^2 ds \right)^{\frac{q}{2}}\right]
\end{align}
and so, by formula \eqref{eq:bismut} and H\"{o}lder's inequality, we have
\begin{equation}
\begin{split}
|dP^V_t f|(x) &\leq \E\left[ |f|^p(X_{t}(x))\right]^{1/p} \frac{1}{t}\left(\E\left[ \bigg\vert \int_0^{t} \langle Q_s,dB_s\rangle \bigg\vert^q\right]^{1/q} \right.\\
&\quad\quad\quad\left.+\,e^{K^-t} \E\left[\left(\int_0^t (t-s)|dV|(X_s(x))ds\right)^q\right]^\frac{1}{q}\right)\\
&\leq e^{K^-t}  \E\left[ |f|^p(X_{t}(x))\right]^{1/p}\left( C_q^{1/q}  \frac{1}{\sqrt{t}} + \kappa_{V,q}(t)\right)
\end{split}
\end{equation}
for all $t> 0$ and $x \in M$. Thus
\begin{equation}
\| dP^V_t f \|_{p} \leq  e^{K^-t} \left( \frac{C_q^{1/q}}{\sqrt{t}} + \kappa_{V,q}(t)\right)  \left( \int_{M} \E\left[ |f|^p(X_{t}(x))\right] dx \right)^{1/p}
\end{equation}
and the result follows by Fubini's theorem. 
\end{proof}

Theorem \ref{thm:linfest} implies that if the Ricci curvature is bounded below then $dP^Vf$ is bounded on $[\epsilon,\infty) \times M$ for each $\epsilon >0$. A similar observation applies to Theorem \ref{thm:lpest}. Our next step will be to use Theorems \ref{thm:linfest} and \ref{thm:lpest} to obtain quantitative estimates which are uniform across all space and time. The price we pay is the inclusion of a term involving $H f$. These are generalizations of the estimate proved by Cheng, Thalmaier and the author in \cite{Ch-Tho-Tha:2017a} and \cite{Ch-Tho-Tha:2018b}.

\begin{theorem}\label{thm:uniformlinfest}
Suppose $V \in C^1(M)$ is non-negative with $\kappa_V(\delta)<\infty$ for some $\delta>0$ and $\Ric \geq 2K$. Suppose $f \in C^2(M)$ with $f,Hf$ bounded. Then the derivative $dP^V f$ is uniformly bounded on $[0,\infty)\times M$ and moreover
\begin{align}
\|dP^V_tf\|_\infty \leq e^{K^-\delta}\left(\left(\sqrt{\frac{2}{\pi \delta}}+\kappa_V(\delta)\right)\|f\|_\infty + \delta\left(\sqrt{\frac{8}{\pi \delta}}+\kappa_V(\delta)\right)\|H f\|_\infty\right)
\end{align}
for all $t\geq 0$.
\end{theorem}

\begin{proof}
According the forward Kolmogorov equation we have
\begin{equation}
P^V_{\delta+t}f = P^V_tf - \int_0^\delta P^V_s(H P^V_tf)ds
\end{equation}
and therefore
\begin{equation}\label{eq:diffestkol}
|dP^V_tf|(x) \leq |dP^V_{\delta+t}f|(x) + \int_0^\delta |dP^V_s(H P^V_tf)|(x)ds
\end{equation}
for all $x \in M$. By Theorem \ref{thm:linfest} we have
\begin{equation}
|dP^V_{\delta+t}f|(x) \leq  e^{K^-\delta} \|P^V_tf\|_\infty \left( \sqrt{\frac{2}{\pi \delta}}  + \kappa_V(\delta) \right)\leq e^{K^-\delta} \|f\|_\infty \left(\sqrt{\frac{2}{\pi \delta}} + \kappa_V(\delta) \right)
\end{equation}
and similarly
\begin{equation}
|dP^V_s(H P^V_tf)|(x) \leq e^{K^-s} \|H f\|_\infty \left(\sqrt{\frac{2}{\pi s}}  + \kappa_V(s) \right).
\end{equation}
Consequently
\begin{equation}
|dP^V_tf|(x) \leq  e^{K^-\delta}\left(\left(\sqrt{\frac{2}{\pi \delta}}+\kappa_V(\delta)\right)\|f\|_\infty + \left(\sqrt{\frac{8\delta}{\pi}}+ \int_0^\delta \kappa_V(s)ds \right)\|H f\|_\infty\right)
\end{equation}
from which the result follows, since $\kappa_V$ is non-decreasing. 
\end{proof}

\begin{theorem}\label{thm:uniformlpest}
Suppose $V \in C^1(M)$ is non-negative with $\Ric \geq 2K$. Suppose $p>1$, set $q = p/(p-1)$ and assume $\kappa_{V,q}(\delta)<\infty$ for some $\delta>0$. Suppose $f\in C^2(M)$ with $f,Hf \in L^p(M)$. Then the derivative $dP^V f$ is uniformly $L^p$-bounded on $[0,\infty)\times M$ and moreover
\begin{align}
\|dP^V_tf\|_{p} \leq e^{K^-\delta}\left(\left(\frac{C_q^{1/q}}{\sqrt{\delta}}+\kappa_{V,q}(\delta)\right)\|f\|_{p}+ \delta\left(\frac{2C_q^{1/q}}{\sqrt{\delta}}+\kappa_{V,q}(\delta)\right)\|H f\|_{p}\right)
\end{align}
where $C_q$ is the constant from the Burkholder-Davis-Gundy inequality.
\end{theorem}

\begin{proof}
By \eqref{eq:diffestkol} and Minkowski's inequality we have
\begin{equation}
\| dP^V_{t}f \|_{p} \leq \| dP^V_{t+\delta}f \|_{p} + \int_0^\delta \| dP^V_s(H P_t^Vf)\|_{p} ds.
\end{equation}
By Theorem \ref{thm:lpest} we have
\begin{equation}
\| dP^V_{t+\delta} f \|_{p} \leq e^{K^-\delta} \left( \frac{C_q^{1/q}}{\sqrt{\delta}} + \kappa_{V,q}(\delta)\right) \| f\|_{p}
\end{equation}
and similarly
\begin{equation}
\| dP^V_s(H P_t^Vf)\|_{p} \leq e^{K^-s} \left( \frac{C_q^{1/q}}{\sqrt{s}} + \kappa_{V,q}(s)\right) \|H f \|_{p}
\end{equation}
from which the result follows, as is the case $p=\infty$. 
\end{proof}

\section{Characterizations of Ricci curvature}\label{sec:characters}

It is well known that Ricci curvature can be characterized in terms of the gradient of the heat semigroup; see for example \cite[Theorem~2.2.4]{wangbook}. Next we show that this characterization extends to the Feynman-Kac semigroups considered above. This shows a way in which the derivative $dV$, appearing in the gradient estimates of the previous section, occurs naturally:

\begin{theorem}\label{thm:charac}
Suppose $V\in C^1(M)$ is bounded below. Let $x\in M$ and $X \in T_xM$ with $|X|=1$. Suppose $f \in C^\infty_0(M)$ with $\nabla f = X$, $\Hess f(x) = 0$ and set $\alpha := f(x)$. Then for any $p>0$ we have
\begin{equation}
\lim_{t\downarrow 0} \frac{P^V_t|\nabla f|^p(x) - |\nabla P^V_tf|^p(x)}{pt} = \frac{1}{2}\Ric(X,X) +\left(1-\frac{1}{p}\right)V + \alpha dV(X).
\end{equation}
\end{theorem}

\begin{proof}
By Taylor expansion at the point $x$ we have
\begin{equation}
P^V_t|\nabla f|^p = |\nabla f|^p + t(\tfrac{1}{2}\Delta - V)|\nabla f|^p + o(t)
\end{equation}
and also
\begin{equation}
|\nabla P_t^Vf|^p = |\nabla f|^p + pt|\nabla f|^{p-2} \langle \nabla (\tfrac{1}{2}\Delta - V)f,\nabla f\rangle + o(t)
\end{equation}
for small $t >0$. Furthermore at the point $x$ we have
\begin{equation}
(\tfrac{1}{2}\Delta-V)|\nabla f|^{p} = \frac{p}{2}|\nabla f|^{p-2} \left(\frac{1}{2}\Delta - \frac{2}{p}V\right)|\nabla f|^2 + \frac{p}{2}\left(\frac{p}{2}-1\right) |\nabla f|^{p-2} |\nabla |\nabla f|^2|^2
\end{equation}
with $|\nabla |\nabla f|^2|^2(x) = |2 (\Hess f(x))^\sharp(\nabla f)|^2 = 0$, and by the Bochner formula
\begin{equation}
\frac{1}{2}\Delta |\nabla f|^2 = \langle \nabla \Delta f,\nabla f \rangle + \Ric(\nabla f,\nabla f)
\end{equation}
so therefore
\begin{align}
\lim_{t\downarrow 0} \frac{P^V_t|\nabla f|^p - |\nabla P^V_tf|^p}{pt} =\text{ }& \frac{1}{2}\Ric(X,X) -\frac{1}{p}V + \langle \nabla (Vf),\nabla f\rangle\\
=\text{ }&\frac{1}{2}\Ric(X,X) +\left(1-\frac{1}{p}\right)V + \alpha dV(X)
\end{align}
as required. 
\end{proof}

In the Section \ref{sec:gradests} we considered uniform estimates on $dP^Vf$ which involved either the supremum norm of $f$ and a constant which diverges for small time, or the supremum norms of both $f$ and $Hf$ and a constant which does not. In light of the previous theorem, we would also like to consider derivative estimates for functions belonging to $C_b^1(M)$. To do so we use the following derivative formula:

\begin{theorem}\label{thm:derivform}
Suppose $V \in C^1(M)$ is bounded below with $\kappa_{V}(t)<\infty$. Suppose the Ricci curvature of $M$ is bounded below with $f\in C_b^1(M)$. Then
\begin{equation}\label{eq:derivformeq}
(dP^V_tf)(v) = \E\left[ \V_t^x ((df)_{X_t(x)}(//_tQ_t v) -f(X_t(x))\int_0^t (dV)_{X_s(x)}(//_s Q_s v)ds)\right]
\end{equation}
for all $x \in M$ and $v \in T_xM$.
\end{theorem}

\begin{proof}
First suppose $f \in C^\infty_0(M)$. Setting $f_s:= P^V_{t-s}f$ and $N_s(v) := dP^V_{t-s}f(//_s Q_s v)$, using the definition of $Q$ as the solution to equation \eqref{eq:eqforQ} and by It\^{o}'s formula and the Weitzenb\"{o}ck formula
\begin{equation}
d \Delta f = \tr \nabla^2 df - df({\Ric}^\sharp)
\end{equation}
we see
\begin{equation}
\begin{split}
dN_s(v) \stackrel{m}{=}\text{ }& df_s(//_s \partial_s Q_sv)ds+(\partial_s df_s)(//_s Q_sv)dt + \frac{1}{2}\tr \nabla^2 (df_s)(//_s Q_sv)ds \\
=\text{ }& V(X_s(x)) N_s(v) ds +f_s(X_s(x)) dV(//_s Q_sv)ds
\end{split}
\end{equation}
where $\stackrel{m}{=}$ denotes equality modulo the differential of a local martingale. It follows that
\begin{equation}
d(\V^x_s N_s(v) ) \stackrel{m}{=} \V^x_s f_s(X_s(x)) dV(//_s Q_sv)ds
\end{equation}
so that
\begin{equation}\label{eq:locmartderform}
\V_s^x (dP^V_{t-s}f)_{X_s(x)}(//_sQ_s v)- \int_0^s \V_r^x P^V_{t-r}f(X_r(x))(dV)_{X_r(x)}(//_rQ_rv)dr
\end{equation}
is a local martingale. To verify that it is a true martingale, we see that
\begin{align}
& \E\left[ \sup_{s\in[0,t]} \left(\V_s^x (dP^V_{t-s}f)_{X_s(x)}(//_sQ_s v)- \int_0^s \V_r^x P^V_{t-r}f(X_r(x)) (dV)_{X_r(x)}(//_rQ_rv)dr\right)\right]\\[5pt]
& \quad \quad \quad \leq \text{ } e^{K^-t} |v| \left(\|dP^Vf\|_{L^\infty([0,t]\times M)} + \|f\|_\infty \kappa_V(t)\right)
\end{align}
which is finite, by Theorem \ref{thm:uniformlinfest}. Formula \eqref{eq:derivformeq} follows by evaluating the expected value of \eqref{eq:locmartderform} at times $0$ and $t$, extending to all $f \in C_b^1(M)$ by approximation. 
\end{proof}

Combining Theorems \ref{thm:charac} and \ref{thm:derivform}, we obtain the following equivalence:

\begin{theorem}\label{thm:gradestbas}
Suppose $V\in C^1(M)$ is bounded below with $\nabla V$ bounded. Let $K \in \mathbb{R}$. Then the following are equivalent:
\begin{itemize}
\item[\normalfont{(1)}] $\Ric \geq 2K$;
\item[\normalfont{(2)}] if $f\in C_b^1(M)$ then
\begin{equation}\label{eq:finalgradest} 
|\nabla P^V_tf| \leq e^{-Kt}P^V_t|\nabla f| + \|\nabla V\|_\infty\left(\frac{1-e^{-Kt}}{K}\right)P_t^V |f|
\end{equation}
for all $t \geq 0$.
\end{itemize}
\end{theorem}

\begin{proof}
That (1) implies (2) follows immediately from Theorem \ref{thm:derivform}. To prove that (2) implies (1), suppose $x \in M$ and $X \in T_xM$ with $|X|=1$ and choose $f \in C^\infty_0(M)$ so that $f(x)=0$, $\nabla f = X$ and $\Hess f(x) = 0$. Then, by applying Theorem \ref{thm:charac} to the case $p=1$, we obtain
\begin{align}
\frac{1}{2}\Ric(X,X) =\text{ }& \lim_{t\downarrow 0} \frac{P^V_t|\nabla f|(x) - |\nabla P^V_tf|(x)}{t}\\
\geq \text{ }& \lim_{t\downarrow 0} \left(\frac{1-e^{-Kt}}{Kt}\right)\left( K P^V_t|\nabla f|(x) - \|\nabla V\|_\infty P_t^V |f|(x)\right)\\
= \text{ }& K |\nabla f|(x) - \|\nabla V\|_\infty |f(x)|\\
=\text{ }& K
\end{align}
as required. 
\end{proof}

For the case $V=0$ it is well known that $\Ric \geq 2K$ if and only if \eqref{eq:finalgradest} holds. So in fact $\Ric \geq 2K$ so long as \eqref{eq:finalgradest} holds for \textit{some} $V\in C^1(M)$ which is bounded below with $\nabla V$ bounded. Moreover, if \eqref{eq:finalgradest} holds for \textit{some} such $V$ then it must hold for \textit{all} such $V$.

\section{Harnack inequalities}\label{sec:harn}

Denoting by $\rho$ the Riemannian distance on $M$, we have also the following Harnack inequality for $P^V_tf$, the proof of which is based on that for the $V=0$ case \cite[Theorem~2.3.4]{wangbook}:

\begin{theorem}\label{thm:harnack}
Suppose $V\in C^1(M)$ is non-negative with $\nabla V$ bounded and $\Ric \geq 2K$. Then for all bounded measurable functions $f \geq 0$ and $p>1$ we have
\begin{equation}
(P^V_tf)^p(x) \leq (P^V_tf^p)(y)\exp\left[\frac{p\rho^2(x,y)}{2(p-1)C_1(t,K)t} +\frac{t\rho(x,y)\|\nabla V\|_\infty}{2C_2(t,K)}\right]
\end{equation}
for all $x,y \in M$ and $t> 0$, where
\begin{equation}
C_1(t,K) := \frac{e^{2Kt}-1}{2Kt}, \quad C_2(t,K) := \frac{Kt}{2} \coth\left(\frac{Kt}{2}\right).
\end{equation}
\end{theorem}

\begin{proof}
Suppose first that $f \in C^2(M)$ with $f$ bounded, $\inf f >0$ and $f$ constant outside a compact set. Given $x \neq y$ and $t>0$, let $\gamma :[0,t]\rightarrow M$ be a geodesic from $x$ to $y$ of length $\rho(x,y)$. Let $v_s := \dot\gamma_s$, so that $|v_s| = \rho(x,y)/t$. Let
\begin{equation}
h(s) := t\left(\frac{e^{2Ks}-1}{e^{2Kt}-1}\right)
\end{equation}
for $s \in [0,t]$, so that $h(0)=0$ and $h(t)=t$. Let $y_s := \gamma_{h(s)}$ and define
\begin{equation}
\phi(s) := \log P^V_s((P^V_{t-s}f)^p)(y_s)
\end{equation}
for $s\in [0,t]$. Then
\begin{equation}
\begin{split}
&\frac{d}{ds}\phi(s)\\
=\text{ }& \frac{1}{P^V_s(P^V_{t-s}f)^p}\bigg(\left(\frac{d}{ds}P^V_s\right)(P^V_{t-s}f)^p + P^V_s\left(\frac{d}{ds}(P^V_{t-s}f)^p\right) \\
&\hspace{65mm}  + \langle \nabla P^V_s(P^V_{t-s}f)^p,\dot{y}_s\rangle\bigg)(y_s)\\
=\text{ }& \frac{1}{P^V_s(P^V_{t-s}f)^p}\bigg((\tfrac{1}{2}\Delta-V)P^V_s(P^V_{t-s}f)^p-p P^V_s((P^V_{t-s}f)^{p-1} (\tfrac{1}{2}\Delta-V)P^V_{t-s}f )\\
&\hspace{63mm}  + \langle \nabla P^V_s((P^V_{t-s}f)^p),\dot{y}_s\rangle\bigg)(y_s)\\
=\text{ }& \frac{1}{P^V_s(P^V_{t-s}f)^p}\bigg(P^V_s(\tfrac{1}{2}\Delta(P^V_{t-s}f)^p)-p P^V_s((P^V_{t-s}f)^{p-1} \tfrac{1}{2}\Delta P^V_{t-s}f )\\
&\hspace{31mm} + (p-1)P^V_s(V(P^V_{t-s}f)^p) + \langle \nabla P^V_s(P^V_{t-s}f)^p,\dot{y}_s\rangle\bigg)(y_s).
\end{split}
\end{equation}
Since
\begin{equation}
\Delta(P^V_{t-s}f)^p = p (P^V_{t-s}f)^{p-1} \Delta P^V_{t-s}f + p(p-1)(P^V_{t-s}f)^{p-2}|\nabla P^V_{t-s}f|^2,
\end{equation}
it follows from Theorem \ref{thm:gradestbas} that
\begin{align}
&\frac{d}{ds}\phi(s)\\
=\text{ }& \frac{1}{P^V_s(P^V_{t-s}f)^p}\bigg( P^V_s \left(\frac{p(p-1)}{2}(P^V_{t-s}f)^{p-2}|\nabla P^V_{t-s}f|^2\right)\\
&\hspace{30mm} + (p-1)P^V_s(V(P^V_{t-s}f)^p) + \langle \nabla P^V_s((P^V_{t-s}f)^p),\dot{y}_s\rangle\bigg)(y_s)\\
\geq\text{ }& \frac{1}{P^V_s(P^V_{t-s}f)^p}P^V_s\bigg((P^V_{t-s}f)^p\bigg( \frac{p(p-1)}{2}|\nabla \log P^V_{t-s}f|^2+ (p-1)V\\
&\hspace{40mm} -\frac{p\rho(x,y)}{t}h'(s)e^{-Ks}|\nabla \log P^V_{t-s}f|\\[2mm]
&\hspace{45mm} -\frac{\rho(x,y)}{t}h'(s)\|\nabla V\|_\infty\left(\frac{1-e^{-Ks}}{K}\right) \bigg)\bigg)(y_s)\\
\geq\text{ }& -\frac{p\rho^2(x,y)h'(s)^2e^{-2Ks}}{2(p-1)t^2}-\frac{\rho(x,y)}{t}h'(s)\|\nabla V\|_\infty\left(\frac{1-e^{-Ks}}{K}\right).
\end{align}
Since
\begin{equation}
h'(s) = 2Kt\left(\frac{e^{2Ks}}{e^{2Kt}-1}\right)
\end{equation}
we have
\begin{equation}
\begin{split}
\frac{d}{ds}\phi(s) \geq-\frac{2p\rho^2(x,y)K^2 e^{2Ks}}{(p-1)(e^{2Kt}-1)^2}-2\rho(x,y)\left(\frac{e^{2Ks}-e^{Ks}}{e^{2Kt}-1}\right)\|\nabla V\|_\infty
\end{split}
\end{equation}
which integrated between $0$ and $t$ yields the inequality. By approximations and the monotone class theorem, as in \cite[Theorem~2.3.3]{wangbook}, this extends to all non-negative, bounded measurable functions. 
\end{proof}

\section{Shift-Harnack inequalities}\label{sec:shiftharn}

In this section we prove two further differentiation formulae, which complement formula \eqref{eq:bismut}, and use them to deduce shift-Harnack inequalities, first introduced by Wang in \cite{Wangharn} and similar to Theorem \ref{thm:harnack} except that the shift in space variable takes place \textit{inside} the semigroup. The approach will be similar to that of \cite{ThalmaierThompson2018}. In particular, we start by supposing that $\alpha$ is a bounded continuously differentiable $1$-form and $\alpha_s$ a solution to the equation
\begin{equation}\label{eq:alphaeq}
\frac{\partial}{\partial s} \alpha_s = (\tfrac{1}{2}\Delta - V)\alpha_s
\end{equation}
for $s \in (0,t]$ with $\alpha_0 = \alpha$. Here $\Delta$ denotes the Hodge Laplacian $-(d^\star + d)^2$ acting on $1$-forms, were $d^\star$ denotes the codifferential operator.

\begin{proposition}
Suppose $V \in C^1(M)$ is bounded below and that $h_s$ is a bounded adapted process with paths belonging to the Cameron-Martin space $L^{1,2}([0,t];[0,\infty))$. Then
\begin{align}
\V_s^x d^\star \alpha_{t-s}(X_s(x))h_s + \V_s^x \alpha_{t-s}\left(//_s Q_s \left(\int_0^s \dot{h}_r Q^{-1}_r dB_r\right)\right)\\
+ \int_0^s \V_r^x \langle dV,\alpha_{t-r}\rangle(X_r(x))h_r dr \label{eq:localmartdiv}
\end{align}
is a local martingale.
\end{proposition}

\begin{proof}
Since $d^\star$ commutes with $\Delta$, with
\begin{equation}
-d^\star(V \alpha_s) = -V d^\star \alpha_s + \langle dV,\alpha_s \rangle,
\end{equation}
by equation \eqref{eq:alphaeq} we have
\begin{equation}
\frac{\partial}{\partial s} d^\star \alpha_s = (\tfrac{1}{2}\Delta - V) d^\star \alpha_s + \langle dV,\alpha_s\rangle.
\end{equation}
Consequently, by It\^{o}'s formula, we find that
\begin{equation}
n_s:=\V_s^x (d^\star \alpha_{t-s})(X_s(x)) + \int_0^s\V_r^x \langle dV,\alpha_{t-r}\rangle(X_r(x)) dr
\end{equation}
is a local martingale and therefore so is
\begin{equation}\label{eq:localmartonediv}
n_s h_s - \int_0^s n_r \dot{h}_r dr
\end{equation}
with the latter starting at $d^\star \alpha_t h_0$. Moreover
\begin{align}
\V_s^x (d^\star \alpha_{t-s})(X_s(x))\dot{h}_s ds &= -\V_s^x \sum_{i=1}^n (\nabla_{//_s e_i} \alpha_{t-s})(//_s e_i) \dot{h}_s ds\\
&=  -\V_s^x \Big\langle \sum_{i=1}^n (\nabla_{//_s e_i} \alpha_{t-s})(//_s Q_s) dB_s^i,\dot{h}_s Q^{-1}_s dB_s\Big\rangle
\end{align}
where $\lbrace e_i \rbrace_{i=1}^n$ is any orthonormal basis of $T_xM$. Since
\begin{equation}\label{eq:dagger}
d(\V_s^x \alpha_{t-s}(//_s Q_s)) = \V_s^x \sum_{i=1}^n (\nabla_{//_s e_i} \alpha_{t-s})(//_s Q_s) dB_s^i
\end{equation}
it follows that
\begin{equation}\label{eq:localmarttwodiv}
\int_0^s \V_r^x (d^\star \alpha_{t-r})(X_r(x))\dot{h}_rdr +\V_s^x \alpha_{t-s}(//_s Q_s)\left(\int_0^s \dot{h}_r Q^{-1}_r dB_r\right)
\end{equation}
is also a local martingale. Using the fact that \eqref{eq:localmartonediv} and \eqref{eq:localmarttwodiv} are local martingales, and since by integration by parts
\begin{align}
-\int_0^s \int_0^r \V^x_u \langle dV,\alpha_{t-u}\rangle(X_u(x)) du \dot{h}_r dr =\,& - \int_0^s \V^x_r \langle dV, \alpha_{t-r}\rangle (X_r(x))dr h_s \\
&+ \int_0^s \V_r^x \langle dV,\alpha_{t-r}\rangle(X_r(x)) h_r dr,
\end{align}
we easily verify that \eqref{eq:localmartdiv} is also a local martingale. 
\end{proof}

\begin{theorem}\label{thm:divthm}
Suppose $V \in C^1(M)$ is bounded below with $\nabla V$ and $\Ric$ bounded. Suppose $\alpha$ is a bounded continuously differentiable $1$-form with $d^\star \alpha$ bounded. Then
\begin{equation}
P^V_t (d^\star \alpha)(x) = -\frac{1}{t}\E\left[\V_t^x \alpha\left(//_t Q_t \int_0^t Q^{-1}_s (dB_s +(t-s)//_s^{-1}\nabla Vds)\right)\right]
\end{equation}
for all $x \in M$ and $t>0$.
\end{theorem}

\begin{proof}
According to the assumptions of the theorem, with $h_s = (t-s)/t$, the local martingale \eqref{eq:localmartdiv} is a true martingale. Furthermore, by \eqref{eq:dagger}, it follows that
\begin{equation}
\alpha_t = \E\left[ \V_t \alpha(//_t Q_t)\right]
\end{equation}
so the result follows by evaluating the martingale \eqref{eq:localmartdiv} at the times $0$ and $t$, and applying the Markov property to the term involving $\nabla V$. 
\end{proof}

Given a vector field $Y$, the Bismut formula \eqref{eq:bismut} provides a probabilistic expression for the derivative $Y(P^V_tf)$ which does not involve derivatives of $f$. By Theorem \ref{thm:divthm}, and using the facts that $\divv Y = - d^\star Y^\flat$ and $\divv (f Y) = Yf + f \divv Y$, we obtain the following theorem which provides a similar expression for the derivative $P^V_t(Y(f))$:

\begin{theorem}\label{thm:diffformula}
Suppose $V \in C^1(M)$ is bounded below with $\nabla V$ and $\Ric$ bounded. Suppose $f$ is a bounded $C^1$ function and $Y$ a bounded $C^1$ vector field for which $\divv Y$ and $Y(f)$ are also bounded. Then
\begin{align}\label{eq:diffformula}
&P_t^V(Y(f))(x)\\[2mm]
=& - \E \left[ \V_t^x f(X_t(x)) (\divv Y)(X_t(x))\right]\\
& +\frac{1}{t} \E\left[\V_t^x f(X_t(x))\bigg\langle Y(X_t(x)), //_t Q_t \int_0^t Q_s^{-1}(dB_s +(t-s)//_s^{-1}\nabla V ds)\bigg\rangle \right]
\end{align}
for all $x \in M$ and $t>0$.
\end{theorem}

By Theorem \ref{thm:diffformula}, we have the following two propositions:

\begin{proposition}\label{prop:firstdivest}
Suppose $V \in C^1(M)$ is non-negative with $\nabla V$ bounded and $2K \leq \Ric \leq 2L$ for constants $K$ and $L$. Suppose $f$ is a bounded $C^1$ function and $Y$ a bounded $C^1$ vector field for which $\divv Y$ and $Y(f)$ are also bounded. Then
\begin{equation}
|P_t^V(Y(f))|(x)\leq \alpha^V_t(Y)(P_t^Vf^2)^\frac{1}{2}(x)
\end{equation}
for all $x \in M$ and $t>0$, where
\begin{align}
\alpha^V_t(Y) := \text{ }&\|\divv Y\|_\infty+ \|Y\|_\infty \frac{e^{-Kt}}{\sqrt{t}}\left( \frac{e^{2Lt}-1}{2Lt}\right)^\frac{1}{2}\\
&+\|Y\|_\infty \|\nabla V\|_\infty \frac{e^{-Kt}}{L}\left( \frac{1}{\frac{Lt}{2}\left(\coth\left(\frac{Lt}{2}\right)-1\right)}-1\right).
\end{align}
\end{proposition}

\begin{proof}
By Theorem \ref{thm:diffformula}, we have
\begin{align}
&|P_t^V(Y(f))|(x)\\[1mm]
\leq \text{ }& \|Y\|_\infty \frac{e^{-tK}}{t} \E\left[ \left( \int_0^t Q_s^{-1}dB_s\right)^2\right]^\frac{1}{2} (P_t^Vf^2)^\frac{1}{2}(x) \\[1mm]
 \text{ }& +\left(\|\divv Y\|_\infty +\|Y\|_\infty \|\nabla V\|_\infty \frac{e^{-tK}}{t}\int_0^t e^{Ls}(t-s) ds\right)(P_t^Vf^2)^\frac{1}{2}(x)\\[1mm]
= \text{ }&\alpha^V_t(Y)(P_t^Vf^2)^\frac{1}{2}(x)
\end{align}
as required. 
\end{proof}

\begin{proposition}\label{prop:seconddivest}
Suppose $V \in C^1(M)$ is non-negative with $\nabla V$ bounded and $2K \leq \Ric \leq 2L$ for constants $K$ and $L$. Suppose $f>0$ is a bounded $C^1$ function and $Y$ a bounded $C^1$ vector field for which $\divv Y$ and $Y(f)$ are also bounded. Then
\begin{equation}
|P_t^V(Y(f))|(x) \leq \delta(P^V_t(f \log f)(x) - P^V_tf\log P^V_tf(x)) + \beta_t^V(\delta,Y) P_t^Vf(x)
\end{equation}
for all $x \in M$, $t>0$ and $\delta >0$, where
\begin{align}
\beta_t^V(\delta,Y):= \text{ }&\|\divv Y\|_\infty +\|Y\|_\infty^2 \frac{e^{-2Kt}}{2t\delta} \left(\frac{e^{2Lt}-1}{2Lt}\right)\\[2mm]
&+ \|Y\|_\infty \|\nabla V\|_\infty \frac{e^{-Kt}}{L}\left( \frac{1}{\frac{Lt}{2}\left(\coth\left(\frac{Lt}{2}\right)-1\right)}-1\right).
\end{align}
\end{proposition}

\begin{proof}
By Theorem \ref{thm:diffformula} and \cite[Lemma~6.45]{Stroock2000}, the latter being essentially Jensen's inequality, we have
\begin{align}
&|P_t^V(Y(f))|(x)\\[3mm]
\leq \text{ } &\|\divv Y\|_\infty P_t^Vf(x)\\[3mm]
&+\bigg|\frac{1}{t} \E\left[\V_t f(X_t(x)) \bigg\langle Y(X_t(x)), //_t Q_t \int_0^t Q_s^{-1}(dB_s +(t-s)//_s^{-1}\nabla V ds)\bigg\rangle \right]\bigg|\\[3mm]
\leq \text{ } & \delta(P^V_t(f \log f)(x) - P^V_tf(x)\log P^V_tf(x))\\[3mm]
&+ \delta P_t^Vf(x) \log \E \left[ \exp\left[\|Y\|_\infty \frac{e^{-Kt}}{t\delta}\bigg|\int_0^t Q_s^{-1}dB_s\bigg| \right]\right]\\[3mm]
& + \left(\|\divv Y\|_\infty + \|Y\|_\infty \|\nabla V\|_\infty \frac{e^{-Kt}}{t}\int_0^t e^{Ls}(t-s)ds \right) P_t^Vf(x)\\[3mm]
\leq \text{ } & \delta(P^V_t(f \log f)(x) - P^V_tf\log P^V_tf(x)) + \beta_t^V(\delta,Y) P_t^Vf(x)
\end{align}
as required. 
\end{proof}

These estimates can be used to derive shift-Harnack inequalities, as introduced by Wang for Markov operators on a Banach space; \cite[Proposition~2.3]{Wangharn}. In particular, suppose as in \cite{ThalmaierThompson2018} that $\lbrace F_s\colon s \in [0,1] \rbrace$ is a $C^1$ family of diffeomorphisms of $M$ with $F_0 = \id_M$. For each $s \in [0,1]$ define a vector field $Y_s$ on $M$ by
\begin{equation}
Y_s := (DF_s)^{-1} \frac{d}{ds}F_s
\end{equation}
and assume that $Y_s$ and $\divv Y_s$ are uniformly bounded.

\begin{theorem}\label{thm:shiftharnone}
Suppose $V \in C^1(M)$ is non-negative with $\nabla V$ bounded and $2K \leq \Ric \leq 2L$ for constants $K$ and $L$. Suppose $f\geq 0$ is a bounded measurable function and that $Y_s$ and $\divv Y_s$ are uniformly bounded. Then
\begin{equation}
P^V_tf (x)\leq P^V_t(f\circ F_1)(x) + \left(\int_0^1 (\alpha_t^V)^2(Y_{s})ds\right)^\frac{1}{2} \left(P^V_tf^2\right)^\frac{1}{2}(x)
\end{equation}
for all $x\in M$ and $t\geq 0$, where $\alpha_t^V$ is defined as in Proposition \ref{prop:firstdivest}.
\end{theorem}

\begin{proof}
It suffices to prove for $f$ continuously differentiable. Since
\begin{equation}
\frac{d}{ds}(f\circ F_s) = \nabla_{V_s} (f\circ F_s),
\end{equation}
and following the proof of \cite[Proposition~2.3]{Wangharn}, we see for all $r>0$ that
\begin{equation}
\begin{split}
&\frac{d}{ds}P^V_t\left( \frac{f}{1+rsf}(F_{1-s})\right) \\
&\quad = -rP^V_t\left( \frac{f^2}{(1+rsf)^2}(F_{1-s})\right) - P^V_t\left(Y_{1-s}\left(\frac{f}{1+rsf}(F_{1-s})\right)\right)
\end{split}
\end{equation}
for all $s \in [0,1]$. Applying Proposition \ref{prop:firstdivest} and proceeding as in the proof of \cite[Proposition~2.3]{Wangharn}, integrating over $s\in [0,1]$ and minimizing over $r>0$, we easily obtain the desired inequality. 
\end{proof}

\begin{theorem}\label{thm:shiftharntwo}
Suppose $V \in C^1(M)$ is non-negative with $\nabla V$ bounded and $2K \leq \Ric \leq 2L$ for constants $K$ and $L$. Suppose $f\geq 0$ is a bounded measurable function and that $Y_s$ and $\divv Y_s$ are uniformly bounded. Then
\begin{equation}
(P^V_tf)^p(x) \leq P^V_t(f^p\circ F_1)(x)\exp\left[\int_0^1 \frac{p}{1+(p-1)s}\beta_t^V\left(\frac{p-1}{1+(p-1)s},Y_s\right)ds\right]
\end{equation}
for all $x \in M$, $t \geq 0$ and $p>1$, where $\beta_t$ is defined as in Proposition \ref{prop:seconddivest}.
\end{theorem}

\begin{proof}
As for the previous theorem, it suffices to prove for $f$ continuously differentiable. The result extends to more general $f$ by approximation. Applying Proposition \ref{prop:seconddivest}, as in the proof of \cite[Proposition~2.3]{Wangharn} with $\beta(s):=1+(p-1)s$ for $s \in [0,1]$, we obtain
\begin{equation}
\frac{d}{ds}\log \left(P^V_t\left(f^{\beta(s)}(F_s)\right)(x)\right)^{\frac{p}{\beta(s)}} \geq -\frac{p}{\beta(s)}\beta_t^V\left(\frac{p-1}{\beta(s)},Y_s\right)
\end{equation}
for $s \in [0,1]$. Integrating over $s$ yields the result. 
\end{proof}

Looking ahead, it would now be desirable to weaken our assumptions on $dV$, such as to suppose simply that $dV$ exists in some weak sense and is locally Kato, or indeed to eliminate terms involving $dV$ altogether. To do the latter could however be difficult, since $dV$ appears natually in Theorem \ref{thm:charac}. Assumptions involving $dV$ do appear elsewhere in the literature, such as in the celebrated work of Li and Yau \cite{LiYau1986}.


\begin{thebibliography}{10}

\bibitem{AizenmanSimon1982}
M.~Aizenman and B.~Simon.
\newblock Brownian motion and {H}arnack inequality for {S}chr\"odinger
  operators.
\newblock {\em Comm. Pure Appl. Math.}, 35(2):209--273, 1982.

\bibitem{Bismut1984}
Jean-Michel Bismut.
\newblock {\em Large deviations and the {M}alliavin calculus}, volume~45 of
  {\em Progress in Mathematics}.
\newblock Birkh\"auser Boston, Inc., Boston, MA, 1984.

\bibitem{BMS}
M.~Braverman, O.~Milatovich, and M.~Shubin.
\newblock Essential selfadjointness of {S}chr\"odinger-type operators on
  manifolds.
\newblock {\em Uspekhi Mat. Nauk}, 57(4(346)):3--58, 2002.

\bibitem{Ch-Tho-Tha:2017a}
Li-Juan Cheng, Anton Thalmaier, and James Thompson.
\newblock Quantitative {$C^1$}-estimates by {B}ismut formulae.
\newblock {\em J. Math. Anal. Appl.}, 465(2):803--813, 2018.

\bibitem{Ch-Tho-Tha:2018b}
Li-Juan Cheng, Anton Thalmaier, and James Thompson.
\newblock Uniform gradient estimates on manifolds with a boundary and
  applications.
\newblock {\em J. Anal.Math.Phys.}, 2018.

\bibitem{ElworthyLi1994}
K.~D. Elworthy and X.-M. Li.
\newblock Formulae for the derivatives of heat semigroups.
\newblock {\em J. Funct. Anal.}, 125(1):252--286, 1994.

\bibitem{Guneysu2012}
Batu G\"uneysu.
\newblock On generalized {S}chr\"odinger semigroups.
\newblock {\em J. Funct. Anal.}, 262(11):4639--4674, 2012.

\bibitem{Kato1972}
Tosio Kato.
\newblock Schr\"odinger operators with singular potentials.
\newblock {\em Israel J. Math.}, 13:135--148 (1973), 1972.

\bibitem{LiYau1986}
Peter Li and Shing-Tung Yau.
\newblock On the parabolic kernel of the {S}chr\"{o}dinger operator.
\newblock {\em Acta Math.}, 156(3-4):153--201, 1986.

\bibitem{LiThompson2018}
Xue-Mei Li and James Thompson.
\newblock First order {F}eynman--{K}ac formula.
\newblock {\em Stochastic Process. Appl.}, 128(9):3006--3029, 2018.

\bibitem{PW}
Enrico Priola and Feng-Yu Wang.
\newblock Gradient estimates for diffusion semigroups with singular
  coefficients.
\newblock {\em J. Funct. Anal.}, 236(1):244--264, 2006.

\bibitem{Simon1982}
Barry Simon.
\newblock Schr\"odinger semigroups.
\newblock {\em Bull. Amer. Math. Soc. (N.S.)}, 7(3):447--526, 1982.

\bibitem{Stroock2000}
Daniel~W. Stroock.
\newblock {\em An introduction to the analysis of paths on a {R}iemannian
  manifold}, volume~74 of {\em Mathematical Surveys and Monographs}.
\newblock American Mathematical Society, Providence, RI, 2000.

\bibitem{Thalmaier1997}
Anton Thalmaier.
\newblock On the differentiation of heat semigroups and {P}oisson integrals.
\newblock {\em Stochastics Stochastics Rep.}, 61(3-4):297--321, 1997.

\bibitem{ThalmaierThompson2018}
Anton Thalmaier and James Thompson.
\newblock Derivative and divergence formulae for diffusion semigroups.
\newblock {\em Ann. Probab.}, 2018.

\bibitem{hessianpaperone}
James Thompson.
\newblock Derivatives of {F}eynman-{K}ac semigroups.
\newblock {\em J. Theoret. Probab.}, 2018.

\bibitem{wangbook}
Feng-Yu Wang.
\newblock {\em Analysis for diffusion processes on {R}iemannian manifolds}.
\newblock Advanced Series on Statistical Science \& Applied Probability, 18.
  World Scientific Publishing Co. Pte. Ltd., Hackensack, NJ, 2014.

\bibitem{Wangharn}
Feng-Yu Wang.
\newblock Integration by parts formula and shift {H}arnack inequality for
  stochastic equations.
\newblock {\em Ann. Probab.}, 42(3):994--1019, 2014.

\end{thebibliography}
\end{document}